\documentclass[12pt]{article}
\usepackage{latexsym, amsmath, amsfonts, amsthm, amssymb}
\usepackage{times}
\usepackage{a4wide}
\usepackage{graphicx, tocloft}
\usepackage{mathrsfs}

\def \Vol {\text{\rm Vol}}

\def \RR {\mathbb R}

\def \EE {\mathbb E}

\def \PP {\mathbb P}

\def \eps {\varepsilon}

\def \cD {\mathcal D}

\def \cF {\mathcal F}

\def \cS {\mathcal S}
\def \cB {\mathcal B}

\def \cD {\mathcal D}

\newtheorem{theorem}{Theorem}[section]

\newtheorem{lemma}[theorem]{Lemma}

\newtheorem{proposition}[theorem]{Proposition}
\newtheorem{corollary}[theorem]{Corollary}

\theoremstyle{definition}
\newtheorem{remark}[theorem]{Remark}

\def\myffrac#1#2 in #3{\raise 2.6pt\hbox{$#3 #1$}\mkern-1.5mu\raise 0.8pt\hbox{$
		#3/$}\mkern-1.1mu\lower 1.5pt\hbox{$#3 #2$}}

\def\qed{\hfill $\vcenter{\hrule height .3mm
		\hbox {\vrule width .3mm height 2.1mm \kern 2mm \vrule width .3mm
			height 2.1mm} \hrule height .3mm}$ \bigskip}

\def \susbeteq {\subseteq}

\begin{document}

\title{A polynomial bound in Dvoretzky's theorem}
\author{Boaz Klartag \ \ and \ \ Shahar Moshe}
\date{}
\maketitle

\begin{abstract}
We present a simple proof of the $\eps$-Dvoretzky conjecture, which asserts that the dependence
on the approximation parameter $\eps$ in Dvoretzky's theorem is polynomial
in $1/\eps$. In particular, if $n \geq (C/\eps)^{\ell/2+1}$, then any $n$-dimensional
convex body has, through any given interior point, an $\ell$-dimensional
section that is $\eps$-close to a Euclidean ball. Here, $C > 0$ is a universal constant.
We in fact obtain a sharper dependence on $\eps$.
The proof is probabilistic,
but uses a different probabilistic model from those employed previously. 
We also prove a simultaneous version of our theorem for a finite family of convex bodies with the origin in their interior, yielding a common
linear subspace on which all of the bodies have nearly-Euclidean sections.
Finally, we compare the radius of the approximating Euclidean ball
with familiar geometric parameters, such as the mean widths of the
body and its dual.
\end{abstract}

\section{Introduction}

Dvoretzky's theorem asserts that high-dimensional convex bodies admit approximately-Euclidean sections of a rather large dimension.
The quantitative question is how high the dimension of the convex body must be
in terms of the dimension of the section and the desired accuracy of
approximation by a Euclidean ball.

\medskip In this paper we improve upon the dependence on the approximation parameter $\eps$ in Dvoretzky's theorem, from exponential to polynomial dependence.
Our basic result is the following:

\begin{theorem} Let $K \subseteq \RR^n$ be a convex body containing the origin in its interior. Let $0 < \eps < 1/2$
and $2 \leq \ell \leq n$, and assume that
\begin{equation}  n \geq \left( \frac{C}{\eps}  \right)^{\frac{\ell+1}{2}} |\log \eps|, \label{eq_1133} \end{equation}
where $C > 0$ is a universal constant.
Then there exists an $\ell$-dimensional linear subspace $E \subseteq \RR^n$ and a centered ellipsoid $\cF \subseteq E$ such that
$$ \cF \subseteq K \cap E \subseteq (1 + \eps) \cF. $$
\label{thm1}
\end{theorem}

For example, it was previously known that an $n$-dimensional convex body with central symmetry admits a $3$-dimensional section through its center
which is $\eps$-close to an ellipsoid for $\eps = C \log \log n / \log n$. See (\ref{eq_1347}) below. Theorem \ref{thm1} improves this
estimate to a power law, namely the section is $\eps$-close to an ellipsoid for $\eps = C \sqrt{\log n} / \sqrt{n}$.

\medskip The proof of Theorem \ref{thm1} is probabilistic and rather simple, as we explain in Section \ref{sec2}. 
While the concentration of measure phenomenon is lurking in the background, the actual proof only requires the Brunn-Minkowski inequality in the space of operators and a union bound.

\medskip 
It is well-known that any $\ell$-dimensional centered ellipsoid admits an $\lceil \ell/2 \rceil$-dimensional
section through its center which is an exact Euclidean ball, see e.g. \cite[Lemma~4]{DvoretskyOriginal}. 
Thus, by applying Theorem \ref{thm1} with $2\ell-1$ in place
of $\ell$, we obtain an $\ell$-dimensional section approximated
by a Euclidean ball under the stronger requirement that 
$n \geq (C/\eps)^\ell |\log\eps|$.
A more refined argument allows us to replace the ellipsoid in Theorem \ref{thm1} by a Euclidean ball, without 
any degradation of the bounds from Theorem \ref{thm1}, apart from the value of the universal constant:

\begin{theorem} Under the assumptions of Theorem \ref{thm1} (with a possibly different universal constant $C >0$), there exists an $\ell$-dimensional linear subspace $E \subseteq \RR^n$
and $r > 0$ such that
$$ r B_E \subseteq K \cap E \subseteq (1 + \eps) r B_E, $$
where $B_E = \left \{ x \in E \, ; \, |x| \leq 1 \right \}$ is the unit Euclidean ball in the subspace $E$.
\label{thm3}
\end{theorem}

\medskip Theorem \ref{thm1} resolves a conjecture that the dependence
on $\eps$ in Dvoretzky's theorem should be polynomial in $1/\eps$, see V. Milman \cite{MilmanTopology} or Klartag and Novikov \cite{KN}.
Paouris and Valettas \cite{PV} proved a version of Theorem \ref{thm1} for centrally-symmetric convex bodies (i.e., when $K = -K$), where the expression on the right-hand side of (\ref{eq_1133}) is replaced by
\begin{equation}  \exp \left( \frac{C \ell |\log \eps|}{\eps} \right). \label{eq_1347} \end{equation}
The estimate (\ref{eq_1347}) improves upon  the exponent in the earlier
bound of Schechtman \cite{schechtman} by a factor of order
$|\log\eps|$. Schechtman's bound in turn improved upon\ Gordon's
estimate \cite{G} by a factor of order
$1/(\eps|\log\eps|^2)$ in the exponent. These proofs built upon a highly influential approach
by Milman to Dvoretzky's theorem relying on  the  concentration of measure phenomena, which has found applications in several branches of mathematics  \cite{Mil71, ConjectureStatement}.

\medskip Dvoretzky's theorem was originally conjectured in the 1950s by Grothendieck \cite{Grothendieck} and by others (see \cite{DvoretzkyOriginalAnnouncement}),
and it was proven around  1960 in \cite{DvoretskyOriginal, DvoretskyErratum}. In addition to Milman's proof, there are alternative
arguments in Szankowski \cite{Szan} and in Figiel \cite{figiel}.
All of these proofs are probabilistic in nature
and yield an exponential (or worse)
dependence on $1/\eps$ already when  $\ell = 3$, similarly to the estimate (\ref{eq_1347}).
Polynomial dependence on $1/\eps$ was known to hold in the particular case
where $K$ is assumed to have the symmetries of the cube, see Bourgain and Lindenstrauss \cite{Bourgain}, Tikhomirov \cite{Tikhomirov} and Fresen~\cite{Fresen}.

\medskip
The case $\ell = 2$ was different. In this case, there is a topological proof yielding polynomial dependence on $1/\eps$,
which is attributed to Gromov by  Milman \cite{MilmanTopology}. Burago, Ivanov and Tabachnikov \cite{BBT} analyzed the shortcomings of certain topological approaches to Dvoretzky's theorem for $\ell \geq 3$.
Klartag and Novikov \cite{KN} used a topological argument for proving the existence of a section admitting several approximate symmetries, with a polynomial dependence on $1/\eps$.

\medskip Traditionally, the probabilistic approaches
to Dvoretzky's theorem followed a strategy that can be summarized as follows:
\begin{center}
{\it Apply an appropriate invertible affine map to $K$ (i.e., put $K$ in the right ``position'').  Then, let~$E$ be a random $\ell$-dimensional
subspace, distributed uniformly in the Grassmannian $G_{n, \ell}$.}
\end{center}
In our context, this approach is already problematic in the case of the cube $K = [-1,1]^n$, where it 
yields a suboptimal, exponential dependence on $1/\eps$. The case of the cube is analyzed in Gordon, Litvak, Pajor and Tomczak-Jaegermann \cite{GLPT}, where an interesting approach is proposed for remedying the situation. They showed that in the case of the cube,
polynomial dependence on $1/\eps$ may be obtained if the subspace $E$ takes the form $E = T(\RR^\ell)$ where
\begin{equation}  T x = \left( \langle x, \theta_1 \rangle, \langle x, \theta_2 \rangle,\ldots, \langle x, \theta_n \rangle \right) \qquad \qquad \qquad (x \in \RR^\ell) \label{eq_1203} 
\end{equation}
and $\theta_1,\ldots, \theta_n$ are independent random vectors distributed uniformly in $S^{\ell-1}$
(or alternatively in $B^{\ell}$).
Here, $S^{\ell-1} = \left \{ x \in \RR^\ell \, ; \, |x| = 1 \right \}$ is the unit sphere 
while $B^{\ell} = \left \{ x \in \RR^\ell \, ; \, |x| \leq 1 \right \}$ is the unit ball. 
Note that the law of the random subspace $E = T(\RR^\ell)$ here is  different from the uniform distribution
on the Grassmannian~$G_{n, \ell}$.

\medskip For a general convex body $K \subseteq \RR^n$ containing the origin in its interior, our probabilistic approach for the proof of Theorem \ref{thm1} may be summarized as follows:
\begin{center}
{\it Consider the convex body $\cB$ consisting of all linear maps $T: \RR^{\ell} \rightarrow \RR^{n}$ with $T(B^{\ell}) \subseteq K$. \\ Let $T$ be a random, uniformly distributed element of $\cB$. Set $E = T(\RR^\ell)$.}
\end{center}
This probabilistic model yields a somewhat more ``invariant'' way to sample a random subspace in a finite-dimensional normed space; it circumvents the need to choose a position for $K$. The drawback of this probability distribution is that it is very much tailored
to $K$, rather than being a Haar measure on a homogeneous space as in the traditional approach to Dvoretzky's theorem.

\medskip
In what follows, we write $\RR^{n \times \ell}$ for the vector space of real, $n \times \ell$-matrices. We
identify a matrix $T \in \RR^{n \times \ell}$ with the associated linear map $x \mapsto Tx$ from $\RR^{\ell}$ to $\RR^n$.
We recall that a convex body is a compact, convex set with non-empty interior, and that each convex body in a finite-dimensional
vector space is of finite, non-zero volume.  Theorem \ref{thm1} follows immediately from:

\begin{theorem} Let $K \subseteq \RR^n$ be a convex body containing the origin in its interior, and let $0 < \eps < 1/2$.
Let $1 \leq \ell \leq n$ be such that
\begin{equation}  n \geq \left( \frac{C}{\eps}  \right)^{\frac{\ell+1}{2}} |\log \eps|, \label{eq_1502} \end{equation}
where $C > 0$ is a universal constant.

\medskip
Write $\cB \subseteq \RR^{n \times \ell}$ for the convex body consisting of all $T \in \RR^{n \times \ell}$ with $T(B^{\ell}) \subseteq K$. Let $T \in \RR^{n \times \ell}$ be a random matrix, distributed uniformly in $\cB$. Set $E = T(\RR^{\ell})$, which is almost-surely an $\ell$-dimensional linear subspace.
Consider the centered ellipsoid $\cF = T(B^{\ell})$.  Then, with probability of at least $1/2$,
$$ \cF \subseteq K \cap E \subseteq (1 + \eps) \cF. $$
\label{thm2}
\end{theorem}

In the case where $K = [-1,1]^n$
we have $\cB = (B^{\ell})^n$, and thus $T$ takes the form (\ref{eq_1203})
where $\theta_1,\ldots, \theta_n$ are independent random vectors that are distributed uniformly in $B^{\ell}$.
Thus, for the cube, our probabilistic model coincides with one of the
constructions in Gordon, Litvak, Pajor and Tomczak-Jaegermann
\cite{GLPT}. 

\medskip 
The lower bound $1/2$ plays no particular role in Theorem \ref{thm2}, and it may be replaced by any fixed number in $(0,1)$,
at the expense of modifying the value of the universal constant $C > 0$. In fact, the proof of Theorem \ref{thm2}
yields a probabilistic lower bound  which is at least 
\begin{equation}  1 - \exp \left(-n (c \eps)^{(\ell+1)/2} \right). \end{equation}

\medskip The law of the random subspace $E$ from Theorem \ref{thm2} is highly dependent on the convex body~$K$.
Nevertheless, it is possible to ``fuse'' the two probability distributions on the Grassmannian corresponding
to two convex bodies $K_1, K_2 \subseteq \RR^n$ -- or any finite number of convex bodies --
and obtain a {\it simultaneous} Dvoretzky theorem. See Theorem \ref{thm4} below.
In particular, Theorem \ref{thm3} follows by  ``fusing'' the probability distributions corresponding to $K$
and to a Euclidean ball.

\medskip Let us discuss the optimality of our estimates.
Write $ N( \ell, \eps) $ for the minimal number such that if $n \geq N(\ell, \eps)$
then for any convex body $K \subseteq \RR^n$ with the origin in its interior, the conclusion of Theorem \ref{thm1} holds true.
A lower bound for $N(\ell, \eps)$ may be obtained by standard analysis of the case of the cube $K = [-1,1]^n$, see Lemma \ref{lem_cube_lower_bound} below. This yields the lower bound
\begin{equation}  N( \ell, \eps) \geq \left( \frac{(1+\eps)^2}{\eps(2 + \eps)} \right)^{\frac{\ell-1}{2}},  \label{eq_1332} \end{equation}
which roughly follows by estimating the spherical measure of a cap.
The exponent $(\ell+1)/2$ in (\ref{eq_1133}) is suboptimal
when $\ell = 2$, since the topological argument mentioned above yields
the better bound $C/\sqrt{\eps}$.
We do not know whether the exponent $(\ell+1)/2$ can be improved to
$(\ell-1)/2$ for arbitrary convex bodies when $\ell \geq 3$.

\medskip For the cube, the optimal exponent is $(\ell-1)/2$, since it suffices to choose $\theta_1,\ldots,\theta_n$ in~(\ref{eq_1203}) to form a $c\sqrt{\eps}$-net in $S^{\ell-1}$.
The random sphere construction
of~\cite{GLPT}, in which the $\theta_i$'s are independent and uniformly
distributed on the {\it unit sphere} $S^{\ell-1}$, yields the slightly weaker  bound $(C / \eps)^{(\ell-1)/2} |\log \eps|$
in the case of the cube. This bound is optimal up to a logarithmic factor. In contrast, taking the $\theta_i$ independently
and uniformly from the {\it unit ball} $B^\ell$ yields the weaker exponent $(\ell+1)/2$ in the case of the cube.
Although we have generalized the ball construction of \cite{GLPT} from the case of the cube
to arbitrary convex bodies, we do not know how to generalize the
sphere construction. 

\medskip 
Nevertheless, in the regime of fixed $\eps$ and large $\ell$, we can refine the proof
of Theorem~\ref{thm2} by focusing on the spherical measure of a cap.
This leads to the following:

\begin{proposition}
For any $0 < \eps < 1/2$,
$$
\lim_{\ell \to \infty} N(\ell,\eps)^{1/\ell}
=
\frac{1+\eps}{\sqrt{\eps(2+\eps)}}.
$$
\label{prop_1540}
\end{proposition}

Theorem \ref{thm2} is proved in Section \ref{sec2}. By optimizing this proof in the regime of large $\ell$ we establish Proposition \ref{prop_1540}.
In Section \ref{sec3} we formulate and prove the simultaneous version of Theorem \ref{thm2} and deduce Theorem \ref{thm3}. 
In Section \ref{sec4} we analyze the radius $r$ of the approximating Euclidean ball provided by Theorem \ref{thm3}.

\medskip 
Throughout this paper we write $x \cdot y = \langle x, y \rangle = \sum_{i=1}^n x_i y_i$ for the standard scalar
product between $x, y \in \RR^n$. We write $|x| = \sqrt{\langle x, x \rangle}$. By $\log$ we refer to the natural logarithm,
and $\#(S)$ is the cardinality of the set $S$.
For $A, B \subseteq \RR^n$ and $t \in \RR$ we write $A+B = \{ x +y \, ; \, x \in A, y \in B \}$ for the Minkowski sum and $tA = \{ t x \, ; \, x \in A \}$.
The Grassmannian $G_{n,\ell}$ is the space of all $\ell$-dimensional
linear subspaces of $\RR^n$. We call its unique rotationally-invariant probability measure the uniform probability measure.
The uniform distribution on a convex body is its volume measure
normalized to have a total mass of one.
Throughout this text we use the letters $c, C, \tilde{C}, \tilde{c}$ to denote positive universal constants whose value may change from one line to the next.
These universal constants may be explicitly computed.

\medskip {\it Acknowledgements}. The idea to use the above probability distribution on the Grassmannian stemmed from a discussion between ChatGPT and the second named author.
ChatGPT was additionally used for finding references, for inspecting and improving the proofs and for language editing.
This research was supported by a grant from the Israel Science Foundation (ISF~$\# 674/24$).

\section{A simple proof using the Brunn-Minkowski inequality}
\label{sec2}

In this section we prove Theorem \ref{thm2}.
Let $K \subseteq \RR^n$ be a convex body with the origin in its interior, let $0 < \eps < 1/2$ and $1 \leq \ell \leq n$.
Consider the Minkowski functional
\begin{equation}  \| x \|_K = \inf \left \{ \lambda > 0 \, ; \, x \in \lambda K \right \} \qquad \qquad \qquad (x \in \RR^n). \label{eq_1658} \end{equation}
Since $K$ is convex, we have the triangle inequality $\| x + y \|_K \leq \| x \|_K + \| y \|_K$ for all $x,y \in \RR^n$.

\medskip
A subset $\cS \subseteq S^{\ell-1}$ is a  $\delta$-net
if for any $x \in S^{\ell-1}$ there exists $y \in \cS$ with $|x-y| \leq \delta$. It is well-known (see, e.g. Rogers \cite{rogers} or Lemma \ref{lem_1714} below)
that there exists a subset $\cS \subseteq S^{\ell-1}$ of cardinality
\begin{equation}  N \leq 2 \left(  \frac{5}{\sqrt{\eps / 8}} \right)^{\ell-1} = 2 \left( \frac{200}{\eps} \right)^{\frac{\ell-1}{2}} \label{eq_1756} \end{equation}
which is a $\delta$-net for $\delta = \sqrt{\eps / 8}$. Enumerate the elements of $\cS$ as $v_1,\ldots,v_N \in S^{\ell-1}$.

\medskip
Consider the unit ball in operator norm
\begin{equation}  \cB = \left \{ T: \RR^\ell \to \RR^n \, ; \, T(B^{\ell}) \subseteq K \right \}. \label{eq_212} \end{equation}
In other words,
 $T \in \cB$ if and only if \begin{equation} \| T \|_{op} \leq 1 \label{eq_1636_} \end{equation} where $\| T \|_{op} = \sup_{x \neq 0} \| T x\|_K / |x|$
is the operator norm of $T$.
The set $\cB \subseteq \RR^{n \times \ell}$ is convex, centrally-symmetric and compact. Thus, $\Vol(\cB)$ is a finite, positive number,
where here $\Vol$ is the Lebesgue measure in the $(\ell n)$-dimensional space $\RR^{n \times \ell}$.
For a unit vector $v \in S^{\ell-1}$ set
$$ \cB_v = \left \{ T \in \cB \, ; \, \| Tv \|_K \leq 1 - \frac{\eps}{4} \right \}. $$
Thus $\cB_v$ is the collection of all contractions $T \in \cB$ which furthermore contract the vector $v$ by at least $\eps/4$.
The key observation is that for certain choices of parameters, the convex set $\cB_v$ is much smaller
in volume than the convex set $\cB$ which contains it. Begin with the following lemma:

\begin{lemma} We have
$$ \frac{\cB_{v_1} + \cB_{v_2} + \ldots + \cB_{v_N}}{N} \subseteq \left( 1 - \frac{\eps}{16N} \right) \cB. $$
\label{lem_1708}
\end{lemma}

\begin{proof} We claim that for any $v,w \in S^{\ell-1}$ and $T \in \cB_{v}$,
\begin{equation} \| T w \|_K \leq 1 - \frac{\eps}{8} + \frac{|v-w|^2}{2}. \label{eq_1652} \end{equation}
Indeed, by the triangle inequality for $\| \cdot \|_K$ and the elementary inequality $\sqrt{1 + x} \leq 1 + x/2$,
\begin{align*} 2 \| T w \|_K & \leq \| T v \|_K + \| T(2w-v) \|_K \leq 1 - \eps/4 + |2 w - v| \\ & = 1 - \eps/4 + \sqrt{1 + 2 |v-w|^2} \leq 2 - \eps/4 + |v-w|^2, \end{align*}
proving (\ref{eq_1652}).

\medskip Fix $v \in S^{\ell-1}$.
 It follows from the definition of the net $\cS$ 
that there exists $i=1,\ldots,N$ with $|v - v_i| \leq \sqrt{\eps/8}$. Thus, according to (\ref{eq_1652}), for any $T_i \in \cB_{v_i}$,
\begin{equation}  \| T_i v \|_K \leq 1 - \frac{\eps}{8} + \frac{\eps/8}{2} = 1 - \frac{\eps}{16}. \label{eq_2105} \end{equation}
Consequently, if $T_j \in \cB_{v_j} \subseteq \cB$ for $j=1,\ldots,N$ and $T = (T_1 + T_2 + \ldots + T_N) / N$,
$$ \| T v \|_K = \frac{1}{N} \left \| \sum_{j=1}^N T_j v \right \|_K \leq \frac{1}{N} \left( \| T_i v \|_K + \sum_{j \neq i}  \left \| T_j v \right \|_K \right) \leq \frac{1}{N} \left( 1 - \frac{\eps}{16} + N-1 \right) = 1 - \frac{\eps}{16 N}. $$
We have thus proved that for any $v \in S^{\ell-1}$ and any
\begin{equation} T \in \frac{\cB_{v_1} + \cB_{v_2} + \ldots + \cB_{v_N}}{N} \label{eq_1729} \end{equation}
we have $$ \| T v \|_K \leq 1 - \frac{\eps}{16N}. $$ In other words, for any $T$ as in (\ref{eq_1729}) we have $\| T \|_{op} \leq 1 - \eps / (16N)$. This implies  the conclusion of the lemma. 
\end{proof}

\begin{proposition} For any $v \in S^{\ell-1}$,
$$ \frac{\Vol(\cB_v)}{\Vol(\cB)} \leq \exp \left(-\ell n (c \eps)^{\frac{\ell+1}{2}} \right), $$
where $c > 0$ is a universal constant.
\label{lem_1806}
\end{proposition}

\begin{proof} Observe that for any $v, w \in S^{\ell-1}$,
\begin{equation}  \Vol(\cB_{v} ) = \Vol(\cB_{w} ).
\label{eq_1619} \end{equation}
Indeed, there exists an orthogonal transformation $U: \RR^\ell \to \RR^\ell$ with $U v = w$.
Clearly, $T \in \cB_{w}$ if and only if $T \circ U \in \cB_{v}$. Since $U$ is an orthogonal transformation, the map $T \mapsto T \circ U$ is volume
preserving on $\RR^{n \times \ell}$ and (\ref{eq_1619}) is proven. 
By (\ref{eq_1619}), Lemma \ref{lem_1708} and the Brunn-Minkowski inequality in $\RR^{n \times \ell}$,
$$ \Vol(\cB_{v_1}) = \left( \prod_{j=1}^N \Vol(\cB_{v_j}) \right)^{1/N} \leq \Vol \left( \frac{\cB_{v_1} + \cB_{v_2} + \ldots + \cB_{v_N}}{N} \right)
\leq \Vol \left[ \left( 1 - \frac{\eps}{16N} \right) \cB \right]. $$
Recall that $\dim(\RR^{n \times \ell}) = \ell n$. Therefore,
for any $v \in S^{\ell-1}$,
$$ \frac{\Vol(\cB_v)}{\Vol(\cB)}
= \frac{\Vol(\cB_{v_1})}{\Vol(\cB)} \leq \left( 1 - \frac{\eps}{16N} \right)^{\ell n} \leq \exp \left( -\frac{\eps \ell n}{16 N} \right) \leq \exp \left( -5 \ell n  (\eps / 200)^{(\ell + 1)/2}  \right), $$
where we used (\ref{eq_1756}) in the last passage.
\end{proof}

From Proposition \ref{lem_1806}, the estimate (\ref{eq_1756}) and the union bound, we immediately obtain:

\begin{corollary} Let $T \in \RR^{n \times \ell}$ be a random matrix, distributed uniformly in the convex body $\cB$. Then,
\begin{equation}  \PP \left( T \in \bigcup_{i=1}^N \cB_{v_i} \right) \leq  N \cdot \exp \left( -\ell n (c \eps)^{\frac{\ell+1}{2}} \right) \leq 2 \left( \frac{\tilde{C}}{\eps} \right)^{\frac{\ell-1}{2}} \cdot \exp \left( -\ell n (c \eps)^{\frac{\ell+1}{2}} \right), \label{eq_1810} \end{equation}
where $\tilde{C}, c > 0$ are universal constants.
\label{cor_1841}
\end{corollary}

\begin{proof}[Proof of Theorem \ref{thm2}]
Since $\ell \leq n$ while $\cB \subseteq \RR^{n \times \ell}$ is a convex body, the linear 
map $T: \RR^{\ell} \rightarrow \RR^n$ is almost-surely injective, and hence $E = T(\RR^\ell)$ is almost-surely an $\ell$-dimensional linear subspace.
We may select a large universal constant $C > 0$ in (\ref{eq_1502}) so that
inequality (\ref{eq_1502}) implies that
$$ 1 + \log 2 + \frac{\ell-1}{2} \log \frac{\tilde{C}}{\eps}
\leq \ell n (c \eps)^{\frac{\ell+1}{2}}, $$
where $\tilde{C}, c > 0$ are the universal constants from (\ref{eq_1810}).
Thus, the expression on the right hand side of (\ref{eq_1810}) is at most $1/e$.
Consequently, by Corollary \ref{cor_1841}, with probability of at least
$1 - 1/e \geq 1/2$ the following holds: The operator $T$ belongs to $\cB$, yet
\begin{equation}  T \not \in \bigcup_{i=1}^N \cB_{v_i}. \label{eq_2110}
\end{equation}
Let us fix $T \in \cB$ satisfying (\ref{eq_2110}).
We claim that for any $v \in S^{\ell-1}$,
\begin{equation}
\| T v \|_K \geq \frac{1}{1+\eps}.
\label{eq_2107}
\end{equation}
Indeed, by the net property of $\cS$, there exists
$i = 1, \ldots, N$ with $|v - v_i| \leq \sqrt{\eps/8}$.
Since $T \in \cB \setminus \cB_{v_i}$, by the triangle inequality and  (\ref{eq_1636_}),
\begin{align*}
\| T v \|_K
&\geq 2\| T v_i \|_K - \| T(2v_i-v) \|_K
> 2(1-\eps/4) - |2 v_i - v| \\
& = 2(1-\eps/4) - \sqrt{1 + 2 |v_i -v|^2} \geq 1-\frac{\eps}{2}-|v_i-v|^2
\geq 1-\frac{5\eps}{8}
\geq \frac{1}{1+\eps}.
\end{align*}
Since $E =T(\RR^{\ell})$, we learn from (\ref{eq_2107}) that the ellipsoid $\cF = T(B^{\ell})$ satisfies
$$ \cF \supseteq \frac{1}{1+\eps} (K \cap E). $$
Since $T \in \cB$ we already know that $\cF \subseteq K$ and therefore $\cF \subseteq K \cap E$. To summarize,
$$ \cF \subseteq K \cap E \subseteq (1+\eps)\cF, $$ completing the proof.
\end{proof}

\subsection{Optimizing the proof for fixed $\eps$ and large $\ell$}

\begin{proof}[Proof of Proposition \ref{prop_1540}]
The lower bound follows from (\ref{eq_1332}). For the upper bound, 
we optimize the previous proof, emphasizing the role of the spherical measure of the cap.
Let $0 < \eps < 1/2$. Fix $0 < a < b < c < 1$ with $a \geq 1/(1+\eps)$. 
 Set
$$
\cB_v(b)= \left\{T\in\cB\,;\,\|Tv\|_K\leq b \right \} \qquad \qquad \qquad (\text{for} \ v \in S^{\ell-1}).
$$
This is a family of convex bodies in $\RR^{n \times \ell}$, all of the same volume, as explained in the 
proof of (\ref{eq_1619}). For $w \in S^{\ell-1}$ consider the spherical cap 
$$ C_w = \left \{ v \in S^{\ell-1} \, ; \, v \cdot w \geq c \right \}. $$
For $w \in S^{\ell-1}, v \in C_w$ and $T\in\cB_v(b)$, by the triangle inequality,
\begin{align*}
\|T w\|_K
& \leq \inf_{t\geq0} \left \{  t \| T v \|_K + \| T(w - t v) \|_K \right \} 
\\ &\leq \inf_{t\geq0} \left \{ tb+|w-tv| \right \} 
 \leq \inf_{t\geq0} \left \{ tb+ \sqrt{1 + t^2 - 2 t c}  \right \} 
\leq 1 - \theta(b,c),
\end{align*} 
for some $\theta(b,c) > 0$ depending solely on $b$ and $c$. 
Suppose that with each $v \in S^{\ell-1}$ we associate an operator $T_v \in \cB_v(b)$ in a measurable way. 
Write $\sigma$ for the uniform probability measure on~$S^{\ell-1}$. 
Then for any $w \in S^{\ell-1}$,
\begin{align*}
\left \| \int_{S^{\ell-1}} T_v w \, d \sigma(v)\right \|_K \leq 
 \int_{S^{\ell-1}} \left \| T_v w \right \|_K \, d \sigma(v) \leq 
\sigma \left ( S^{\ell-1} \setminus C_w \right) + \sigma(C_w) \cdot (1 - \theta(b,c)). 
\end{align*} 
Consequently, by averaging via Minkowski addition,\footnote{The integral of convex bodies in (\ref{eq_1431}) is defined either by integrating their support
functions, or else equivalently by taking limits of finite weighted
Minkowski sums.}
\begin{equation} 
\int_{S^{\ell-1}}\cB_v(b)\,d\sigma(v)
\subseteq (1 - \eta_{\ell}) \cB
\label{eq_1431} \end{equation}
where
\begin{equation}  \eta_{\ell} = \sigma(C_v) \cdot \theta(b,c). \label{eq_1515} \end{equation}
 The Brunn-Minkowski inequality in $\RR^{n \times \ell}$ therefore implies that for any $v \in S^{\ell-1}$, 
\begin{equation}
\frac{\Vol(\cB_v(b))}{\Vol(\cB)}
\leq (1 - \eta_{\ell})^{n\ell}
\leq \exp \left( -n\ell \eta_{\ell} \right). 
\label{eq_1504} \end{equation}
Choose a $(b-a)$-net
$\tilde{\cS}\subseteq S^{\ell-1}$ whose cardinality is at most $\left(10 / (b-a) \right)^\ell$.
We claim that
\begin{equation}
N(\ell,\eps)\leq \ell+1+\frac{\log(10/(b-a))}{\eta_{\ell}}. \label{eq_1512}
\end{equation} 
Indeed, assume that $n$ exceeds the expression on the right-hand side of (\ref{eq_1512}). 
Let $T$ be uniformly distributed in $\cB$. By (\ref{eq_1504}) and the union bound,
$$
\PP\left(T \in \bigcup_{v \in \tilde{\cS}} \cB_v(b) \right)
\leq \left( \frac{10}{b-a} \right)^\ell \cdot e^{-n\ell\eta_{\ell}}
\leq \exp\left[\ell\log\left(\frac{10}{b-a}\right)-n\ell\eta_{\ell} \right] < 1,
$$
where in the last passage we used the fact that $n$ exceeds the expression on the right-hand side of~(\ref{eq_1512}). 
Thus there exists $T \in \cB \setminus \bigcup_{v \in \tilde{\cS}} \cB_v(b)$.
For any $u\in S^{\ell-1}$, choose $v\in\tilde{\cS}$ with $|u-v|\leq b-a$.
Then $\|Tu\|_K\geq\|Tv\|_K-|u-v|>a \geq 1/(1+\eps)$. 
Thus, the ellipsoid $\cF=T(B^\ell)$ satisfies
$$
\cF\subseteq K\cap E\subseteq(1+\eps)\cF,
$$
where $E = T(\RR^{\ell})$. This proves (\ref{eq_1512}). 
Next, it follows from (\ref{eq_1515}) 
and a standard computation as in (\ref{eq_1457}) below that 
$$ 
\lim_{\ell\to\infty}\eta_{\ell}^{1/\ell}=
\lim_{\ell \to \infty} \sigma(C_v)^{1/\ell} = \sqrt{1-c^2}
\cdot \lim_{\ell \to \infty}  \left( 
\frac{\int_{B^{\ell-1}} \frac{dy}{\sqrt{1-(1-c^2)|y|^2}}}
      {\int_{B^{\ell-1}} \frac{dy}{\sqrt{1-|y|^2}}} \right)^{1/\ell} = \sqrt{1-c^2}. $$
Consequently, by (\ref{eq_1512}),
\begin{equation}
\limsup_{\ell\to\infty}N(\ell,\eps)^{1/\ell}
\leq\frac{1}{\sqrt{1-c^2}}.
\label{eq_1518} \end{equation}
Since $a < b < c$ were arbitrary numbers in $(0,1)$ with $a \geq 1 / (1+\eps)$, 
the $\limsup$ in (\ref{eq_1518}) is at most $1 / \sqrt{ 1 - (1+\eps)^{-2}}$, 
and the proof is complete.
\end{proof}

\section{Fusing probability distributions on the Grassmannian}
\label{sec3}

In this section we prove a simultaneous version of Theorem \ref{thm2} and deduce Theorem \ref{thm3} as a corollary. 
We begin by constructing the ``$\ell$-balancing scales'' of a finite family of convex bodies
in $\RR^n$. As before, for a convex body $K \subseteq \RR^n$ containing the origin in its interior, we denote
\begin{equation}  \cB = \cB_{\ell}(K) = \left \{ T: \RR^\ell \to \RR^n \, ; \, T(B^{\ell}) \subseteq K \right \}. \label{eq_1539} \end{equation}
Clearly, for any $K_1,\ldots,K_m \subseteq \RR^n$ and $t > 0$,
\begin{equation} \cB_{\ell}(K_1) \cap \cB_{\ell}(K_2) \cap \ldots \cap \cB_{\ell}(K_m) = \cB_{\ell}(K_1 \cap K_2 \cap \ldots \cap K_m) \label{eq_1629} \end{equation}
while  $\cB_{\ell}(t K_1) = t \cB_{\ell}(K_1)$.
We write $\RR^n_+ = \left \{ (x_1,\ldots,x_n) \in \RR^n \, ; \, \forall i, \ x_i > 0 \right \}$.

\begin{lemma} Let $n, m \geq 1, 1 \leq \ell \leq n$, and let $K_1,\ldots,K_m \subseteq \RR^n$ be convex bodies, each containing the origin in its interior.
Consider the optimization problem
$$ \sup_{t_1,\ldots, t_m > 0} \left[ \frac{\Vol \left(  \cB_{\ell} \left(\bigcap_{j=1}^m t_j K_j \right) \right)}{\left( \prod_{j=1}^m t_j \right)^{n \ell/m}} \right]. $$
Then the supremum is attained at some $r_1,\ldots,r_m > 0$. We refer to any $r_1,\ldots,r_m > 0$ attaining the maximum 
as ``$\ell$-balancing scales'' for $K_1,\ldots,K_m \subseteq \RR^n$.
\label{def_1431}
\end{lemma}

\begin{proof} Set $M = n \ell$ and $Q_i = \cB_{\ell}(K_i) \subseteq \RR^{n \times \ell} \cong \RR^M$. Thus $Q_1,\ldots,Q_m \subseteq \RR^M$
are centrally-symmetric convex bodies. Thanks to (\ref{eq_1629}), we need to show that 
the supremum of the function 
\begin{equation}  F(t_1,\ldots,t_m) =  \frac{\Vol \left(  \bigcap_{j=1}^m t_j Q_j \right)}{\left( \prod_{j=1}^m t_j \right)^{M/m}} 
\label{eq_1744} \end{equation}
on the orthant $\RR^m_+$, is attained. Observe that $F$ is $0$-homogeneous, i.e., $F(\lambda t) = F(t)$ for all $t \in \RR^m_+$ and $\lambda > 0$. Hence it suffices 
to show that the supremum of  $F$ on the hypersurface $$ H = \left \{ (t_1,\ldots,t_m) \in \RR^m_+ \, ; \, \prod_{i=1}^m t_i = 1 \right \} $$ is attained.
Note that 
\begin{equation}
 \Vol \left( \bigcap_{j=1}^m t_j Q_j \right) \leq \min_{j=1,\ldots,m} \Vol \left( t_j Q_j \right) \leq \left( \min_j t_j \right)^M \cdot \max_j \Vol(Q_j). 
 \label{eq_1547_} \end{equation}
The function $F$ is clearly positive and continuous on $H$, and it follows from (\ref{eq_1547_}) that $F(t)$ tends to zero 
as $t \to \infty$ in the hypersurface $H$. Hence the maximum of $F$ is attained.
\end{proof}

If $r_1,\ldots,r_m$ are $\ell$-balancing scales for $K_1,\ldots,K_m \subseteq \RR^n$, then 
the same holds for $\lambda r_1,\ldots, \lambda r_m$ for any $\lambda > 0$. We remark that 
$\ell$-balancing scales are not necessarily unique, not even up to a multiplication by a scalar. For example,
consider the case where 
$$
K_1=[-2,2]\times[-1,1] \subseteq \RR^2,
\qquad \text{and} \qquad
K_2=[-1,1]\times[-2,2] \subseteq \RR^2.
$$
A computation shows that in this case, for $\ell=1,2$,
$$
F(1,t)
=
\Vol(B^\ell)^2\min \left \{2t,1,2/t \right\}^{\ell}.
$$
The maximum of this function on $\RR_+$ is not uniquely attained.

\begin{remark} 
By differentiating the function $F$ in (\ref{eq_1744}), we obtain a necessary condition for $\ell$-balancing scales 
in terms of the {\it cone volume measure}. Let $r_1,\ldots,r_m$ be $\ell$-balancing scales, and set
$$
\cB=\cB_\ell\left(\bigcap_{i=1}^m r_iK_i\right),
\qquad
\Gamma_i=\partial\cB\cap\partial\bigl(r_i\cB_\ell(K_i)\bigr).
$$ 
If $F$ is differentiable at $(r_1,\ldots,r_m)$, then the vanishing of the gradient of  $F$ amounts 
to the condition that the cone volume measure on $\partial \cB$ attributes the same amount 
of  mass to all of the $\Gamma_i$'s.
The appearance 
of the cone volume measure and the non-uniqueness example above are also reminiscent of the log-Minkowski problem, 
see B\"or\"oczky, Lutwak, Yang and Zhang \cite{BLYZ, BLYZ2} and references therein.
\end{remark}

A useful feature of $\ell$-balancing scales $r_1,\ldots,r_m$ 
is that assuming that $0 \leq t_j \leq r_j$ for all $j$, the volume of 
$ \cB_{\ell}(\cap_{j=1}^m t_j K_j)$ drops 
substantially when one of the $t_j$ drops below $r_j$.

\begin{lemma} 
Let $n, m \geq 1, 1 \leq \ell \leq n$, and let $K_1,\ldots,K_m \subseteq \RR^n$ be convex bodies, each containing the origin in its interior.
Assume that $1,\ldots,1$ are $\ell$-balancing scales for these convex bodies. Then for any $i=1,\ldots,m$ and $t > 0$,
$$
\Vol \left(     \cB_{\ell} \left( t  K_i \right)     \cap \bigcap_{j \neq i} \cB_{\ell} \left(  K_j \right)
\right) \leq t^{n \ell/m} \cdot \Vol \left(     \bigcap_{j=1}^m \cB_{\ell} \left(  K_j \right) \right).
$$
\label{cor_931}
\end{lemma}

\begin{proof} This follows from the defining property of  $\ell$-balancing scales from Lemma \ref{def_1431}. Indeed, 
by the maximizing property, 
$$
\frac{   \Vol \left(
        \cB_{\ell} \left( t  K_i \right)
        \cap \bigcap_{j \neq i} \cB_{\ell} \left( K_j \right)
    \right)}{t^{n \ell/m}} \leq \frac{   \Vol \left(
\bigcap_{j=1}^m \cB_{\ell} \left( K_j \right)
    \right)}{1^{n \ell/m}}. $$
\end{proof}

The following theorem is the simultaneous version of Theorem \ref{thm2}. 
It shows that when $K = \cap_i (r_i K_i)$ for appropriate $\ell$-balancing scales, the probability distribution 
on the Grassmannian associated with $K$ also supplies nearly-ellipsoidal sections for all of the convex bodies $K_1, \ldots,  K_m$.

\begin{theorem} Let $n, m \geq 1, 1 \leq \ell \leq n, 0 < \eps < 1/2$, and let $K_1,\ldots,K_m \subseteq \RR^n$ be convex bodies, each containing the origin in its interior.
Let $r_1,\ldots,r_m > 0$ be $\ell$-balancing scales for $K_1,\ldots,K_m$.
Assume that
\begin{equation}  n \geq \left( \frac{C}{\eps} \right)^{\frac{\ell+1}{2}} \left( m |\log \eps| + \frac{m \log m}{\ell} \right), \label{eq_1800} \end{equation}
where $C > 0$ is a universal constant. Set $$ K = (r_1 K_1) \cap (r_2 K_2) \cap \ldots \cap (r_m K_m). $$

\medskip Write $\cB \subseteq \RR^{n \times \ell}$ for the convex body consisting of all $T \in \RR^{n \times \ell}$ with $T(B^{\ell}) \subseteq K$. Let $T \in \RR^{n \times \ell}$ be a random matrix, distributed uniformly in $\cB$. Set $E = T(\RR^{\ell})$, 
which is almost-surely an $\ell$-dimensional linear subspace.
Consider the centered ellipsoid $\cF = T(B^{\ell})$.  Then, with probability of at least $1/2$, 
for  $1 \leq i \leq m$ we have
$$ \cF \subseteq (r_i K_i) \cap E \subseteq (1 + \eps) \cF. $$
\label{thm4}
\end{theorem}

\begin{proof} By scaling we may assume that $r_i = 1$ for all $i$ (otherwise, replace $K_i$ by $r_i K_i$).
For $i=1,\ldots,m$ and $v \in S^{\ell-1}$ define
$$ \cB_i(v) = \left \{ T \in \cB \, ; \, \| T v \|_{K_i} \leq 1 - \frac{\eps}{4} \right \} $$
where the Minkowski functional $\| \cdot \|_{K_i}$ is defined as in (\ref{eq_1658}) above. 
By rotational invariance, as explained in the beginning of the proof 
of Proposition \ref{lem_1806}, the volume of $\cB_i(v)$ does not depend on the choice of $v \in S^{\ell-1}$. Set also 
$$ \cB_i = \left \{ T: \RR^\ell \rightarrow \RR^n \, ; \, T(B^{\ell}) \subseteq K_i \right \} $$
so that $$ \cB = \cB_1 \cap \ldots \cap \cB_m. $$ Let $\cS = \{ v_1,\ldots,v_N \} \subseteq S^{\ell-1}$ be a $\delta$-net
for $\delta = \sqrt{\eps/8}$ as in Section \ref{sec2}, with $N$ satisfying the bound (\ref{eq_1756}). We claim that 
for any $i$,
\begin{equation}
\frac{1}{N} \sum_{j=1}^N \cB_i(v_j) \subseteq \left(1 - \frac{\eps}{16N} \right) \cB_i. 
\label{eq_1703}
\end{equation}
Indeed, this follows from Lemma \ref{lem_1708} applied for $K_i$ in place of $K$.
Since $\cB_i(v_j) \subseteq \cB$ we obtain from (\ref{eq_1703}) that
\begin{equation}  \frac{1}{N} \sum_{j=1}^N \cB_i(v_j) \subseteq  \cB \cap \left(1 - \frac{\eps}{16N} \right) \cB_i  = \left(1 - \frac{\eps}{16N} \right) \cB_i \cap \bigcap_{j \neq i} \cB_j. \label{eq_1724} \end{equation}
Since $\Vol( \cB_i(v_j) )$ does not depend on $j$, by the Brunn-Minkowski inequality we conclude from (\ref{eq_1724}) that
$$ \Vol \left( \cB_i(v_1) \right)  \leq \Vol \left( \frac{1}{N} \sum_{j=1}^N \cB_i(v_j) \right) \leq \Vol \left( \left(1 - \frac{\eps}{16N} \right) \cB_i \cap \bigcap_{j \neq i} \cB_j \right). $$
Therefore, by Lemma \ref{cor_931}, for $i=1,\ldots,m$ and $v \in S^{\ell-1}$,
\begin{equation}  \frac{ \Vol \left( \cB_i(v) \right) }{\Vol(\cB)} = \frac{ \Vol \left( \cB_i(v_1) \right) }{\Vol(\cB)} \leq \left(1 - \frac{\eps}{16N} \right)^{n \ell/m}
\leq \exp \left( -\frac{\eps \ell n}{16 m N} \right). \label{eq_1759} \end{equation}
Let $T$ be uniformly distributed in $\cB$. By (\ref{eq_1759}) 
and the union bound,
$$
\PP\left(T\in\bigcup_{i=1}^m\bigcup_{j=1}^N\cB_i(v_j)\right)
\leq m N \exp\left(-\frac{\eps\ell n}{16mN}\right).
$$
Recall from (\ref{eq_1756}) that $N \leq 2 (200/\eps)^{(\ell-1)/2}$. 
We may select a large enough universal constant $C$ 
in (\ref{eq_1800})
so that (\ref{eq_1800}) implies that
$$
1+\log(mN)\leq\frac{\eps\ell n}{16mN}.
$$
Consequently, with probability of at least $1-1/e\geq1/2$,
$$
T\notin\bigcup_{i=1}^m\bigcup_{j=1}^N\cB_i(v_j).
$$
Fix such an operator $T$. Since $T\in\cB\subseteq\cB_i$ for every $i$,
we have $\|Tx\|_{K_i}\leq|x|$ for all $x\in\RR^\ell$.
For any $i=1,\ldots,m$ and $v\in S^{\ell-1}$, choose
$j\in\{1,\ldots,N\}$ with $|v-v_j|\leq\sqrt{\eps/8}$.
As in the proof of Theorem \ref{thm2}, the triangle inequality gives
\begin{align} \label{eq_1802}
\|Tv\|_{K_i}
&\geq 2\|Tv_j\|_{K_i}-\|T(2v_j-v)\|_{K_i} 
>2(1-\eps/4)-|2v_j-v| \\
&=2(1-\eps/4)-\sqrt{1+2|v_j-v|^2} 
\geq1-\frac{\eps}{2}-|v_j-v|^2
\geq1-\frac{5\eps}{8}
\geq\frac{1}{1+\eps}. \nonumber
\end{align}
We learn from (\ref{eq_1802}) that $K_i\cap E\subseteq(1+\eps)T(B^\ell)$ for all $i$, where $E = T(\RR^{\ell})$.
Since $T\in\cB\subseteq\cB_i$, the centered ellipsoid $\cF=T(B^\ell)$ therefore
satisfies
$$
\cF\subseteq K_i\cap E\subseteq(1+\eps)\cF
\qquad (i=1,\ldots,m).
$$
\end{proof}

\begin{proof}[Proof of Theorem \ref{thm3}]
Let $(1,r)$ be $\ell$-balancing scales for $K$ and $B^n$. Applying 
Theorem \ref{thm4} with $m = 2$ we obtain a subspace $E \in G_{n, \ell}$
and a centered ellipsoid $\cF \subseteq E$ such that 
$$
\cF\subseteq K\cap E\subseteq(1+\eps)\cF
\qquad \text{and} \qquad \cF\subseteq r B^n \cap E\subseteq(1+\eps)\cF.
$$
This implies that
$$ \frac{1}{1+\eps} r B_E \subseteq K \cap E \susbeteq (1 + \eps) r B_E. $$
The theorem follows by replacing $\eps$ with $\eps/3$ and slightly adjusting the value of $r$.
\end{proof}

\section{The radius of the approximating Euclidean ball}
\label{sec4}

How can we estimate the radius of the approximating Euclidean ball from Theorem \ref{thm3}? 
In view of the proof of Theorem \ref{thm3}, we need to bound a number $r > 0$ such that $(1,r)$ are $\ell$-balancing scales for $K$ and $B^n$.
Set $$ M(K) = \int_{S^{n-1}} \| x \|_K d \sigma(x) $$ 
where $\sigma$ is the uniform probability measure on $S^{n-1}$, and write $$ M^*(K) = M(K^{\circ}) $$ for half of the mean width of $K$.
Here $K^{\circ} = \left \{ x \in \RR^n \, ; \, \forall y \in K, \, x \cdot y \leq 1 \right \}$ is the polar body.
In this section we show that for all relevant values of $\ell$, the radius $r$ of the approximating Euclidean ball lies in the interval $$ \left [ \frac{c}{M(K)}, C M^*(K) \right] $$ for universal constants $c, C > 0$. Denote the volume radius of $K \subseteq \RR^n$ by 
$$
v.rad.(K)=\left(\frac{\Vol(K)}{\Vol(B^n)}\right)^{1/n}.
$$
This is the radius of the Euclidean ball with the same volume as $K$.

\begin{lemma}
Let $1\leq\ell\leq n$, and let $K\subseteq\RR^n$ be a convex body
containing the origin in its interior. Then
\begin{equation}
\Vol(\cB_\ell(K))
\leq v.rad.(K)^{n\ell} \cdot \Vol(\cB_\ell(B^n)).
\label{eq_operator_rearrangement}
\end{equation}
\label{lem_operator_rearrangement}
\end{lemma}

\begin{proof} This is a special case of Haddad \cite{Haddad}. We include here a short proof for completeness. 
Set $R=v.rad.(K)$ and choose a dense sequence $v_1,v_2,\ldots$ in
$S^{\ell-1}$ whose first $\ell$ elements are linearly independent. 
The rearrangement inequality of Brascamp-Lieb-Luttinger \cite{BLL}
and Rogers \cite{rogers2} implies that for any $q\geq\ell$,
\begin{equation}
\int_{\RR^{n\times\ell}}\prod_{j=1}^q 1_K(Tv_j)\,dT
\leq
\int_{\RR^{n\times\ell}}\prod_{j=1}^q 1_{RB^n}(Tv_j)\,dT,
\label{eq_1535} \end{equation}
where $1_K$ is the indicator function of $K$.
The two integrals in (\ref{eq_1535}) are finite since $v_1,\ldots, v_{\ell} \in \RR^{\ell}$ are linearly independent. 
As $q\to\infty$ these two integrals tend to 
$\Vol(\cB_\ell(K))$ and $\Vol(\cB_\ell(RB^n))$, respectively, by the bounded convergence theorem.
Since $\cB_\ell(RB^n)=R\cB_\ell(B^n)$, the lemma follows.
\end{proof}

\begin{proposition}
Let $1\leq\ell\leq n$, and let $K\subseteq\RR^n$ be a convex body
containing the origin in its interior. Suppose that $(1,r)$ are
$\ell$-balancing scales for $K$ and $B^n$. Then
\begin{equation}
r\leq\inf_{t>0}\frac{v.rad.(K+tB^n)^2}{t}.
\label{eq_radius_parallel_volume}
\end{equation}
In particular:
\begin{enumerate}
\item[(i)] $\displaystyle r \leq 4 M^*(K)$. 
\item[(ii)] If the uniform distribution on $K$ has mean zero and identity covariance (i.e., $K$ is in ``isotropic position''), then
\begin{equation} 
r\leq C \cdot v.rad.(K),
\label{eq_1547}
\end{equation}
where $C>0$ is a universal constant. Moreover, (\ref{eq_1547}) holds 
true whenever $K$ is in Milman position (see e.g. Pisier \cite[Chapter 7]{pisier}), with constant
$C$ depending only on the Milman-position constant.
\end{enumerate}
\label{prop_radius_M_position}
\end{proposition}

\begin{proof} Set $M=n\ell = \dim(\RR^{n \times \ell})$, $\cB=\cB_\ell(K)$ and $\cD=\cB_\ell(B^n)$.
Since $\cB, \cD \subseteq \RR^{n \times \ell}$ are centrally-symmetric convex bodies, by the Rogers-Shephard inequality 
(see \cite[Remark~(1) following Corollary~3]{MilmanPajor}),
\begin{equation}
\Vol( \cB \cap t \cD ) \Vol(\cB + t \cD) \geq \Vol(\cB) \Vol(t \cD). 
\label{eq_2146} \end{equation}
Observe that $\cB+t\cD\subseteq\cB_\ell(K+tB^n)$. 
Lemma \ref{lem_operator_rearrangement} thus gives
$$
\Vol(\cB+t\cD)^{1/M}
\leq v.rad.(K+tB^n) \cdot \Vol(\cD)^{1/M}.
$$
Thus, by the Rogers-Shephard inequality (\ref{eq_2146}),
$$
F(t):=\frac{\Vol(\cB\cap t\cD)^{1/M}}{\sqrt t}
\geq\frac{\sqrt t\,\Vol(\cB)^{1/M}}{v.rad.(K+tB^n)}.
$$
The maximizing property of $r$ now shows that, for every $t>0$,
$$
\frac{\Vol(\cB)^{1/M}}{\sqrt r}
\geq F(r)\geq F(t)
\geq\frac{\sqrt t\,\Vol(\cB)^{1/M}}{v.rad.(K+tB^n)},
$$
which proves (\ref{eq_radius_parallel_volume}). Next, by the Urysohn inequality $v.rad.(K) \leq M^*(K)$. Hence,
by (\ref{eq_radius_parallel_volume}) with $t = M^*(K)$,
$$ r \leq \frac{ ( M^*(K + t B^n) )^2}{t} = 4 \frac{M^*(K)^2}{M^*(K)} = 4 M^*(K), $$
proving conclusion (i). As for conclusion (ii), the 
resolution of the slicing problem in Klartag and Lehec 
\cite{KlartagLehecSlicing} implies that $K$ is in Milman position, with a
Milman-position constant that is at most $C$. Therefore, by basic properties 
of the Milman position (e.g. \cite[Theorem 7.1]{pisier}), for $R = v.rad.(K)$,
$$
v.rad.(K+RB^n)\leq C R,
$$
which implies (ii) thanks to  (\ref{eq_radius_parallel_volume}).
\end{proof}

As explained, e.g., in Kashin and Szarek \cite[Lemma~3]{KashinSzarek},
if $E \subseteq \RR^n$ is an $\ell$-dimensional linear subspace and
$r B_E\subseteq[-1,1]^n$, then $r\leq\sqrt{n/\ell}$.
Thus, even when $K$ is in isotropic position,
we cannot bound the radius of the approximating Euclidean ball from
below by a universal constant times the volume radius of $K$,
even though Proposition \ref{prop_radius_M_position}(ii) guarantees
such an upper bound.

\begin{proposition} Let $1 \leq \ell \leq n$ with $\ell \leq c \log n$ and let $K \subseteq \RR^n$ be a convex body with the origin in its interior.
Suppose that $(1,r)$ are $\ell$-balancing  scales for $K$ and $B^n$. Then,
$$  \frac{\tilde{c}}{M(K)} \leq r, $$
where $c, \tilde{c} > 0$ are universal constants. 
\end{proposition}
\begin{proof}
By scaling, we may assume that $M(K)=1$. 
As before, set $M = n \ell = \dim(\RR^{n \times \ell}), \cB=\cB_\ell(K)$ and $\cD=\cB_\ell(B^n)$. Write
$$ H=K\cap(-K)\cap B^n. $$ 
Then $H \subseteq \RR^n$ is a centrally-symmetric convex body with $M(H) \leq M(K) + M(-K) + 1 = 3$. Note that
\begin{equation} \cB_\ell(H)=\cB\cap\cD. \label{eq_1637} \end{equation}
Write $\| \cdot \|_H$ for the Minkowski 
functional of $H$. Let $G\in\RR^{n\times\ell}$ be a random matrix all of whose entries are independent, standard Gaussian
variables. Let $g$ be a standard Gaussian vector in $\RR^n$. Then for any  $v \in S^{\ell-1}$, 
$$ \EE \| G v \|_H = \EE \| g \|_H = \EE |g| \cdot M(H) \leq 3 \EE |g| \leq 3 \sqrt{ \EE |g|^2 } = 3 \sqrt{n}, $$
where the second equality is the standard relation between Gaussian integrals and $M(H)$, proven by integration 
in polar coordinates. By the Markov-Chebyshev inequality, for any $v \in S^{\ell-1}$, 
\begin{equation}  \PP \left( G v \in 6 \sqrt{n} H \right) \geq \frac{1}{2}. \label{eq_1632} \end{equation}
Let $v_1,\ldots,v_N \in S^{\ell-1}$ be a $1/2$-net in $S^{\ell-1}$, with $N \leq 10^{\ell}$. 
By (\ref{eq_1632}) and Royen's Gaussian correlation inequality \cite[Theorem~1]{LatalaMatlak},
\begin{equation}  \PP \left( \forall i, \ G v_i \in 6 \sqrt{n} H \right) \geq \prod_{i=1}^N \PP \left( G v_i \in 6 \sqrt{n} H \right) \geq \frac{1}{2^N}. \label{eq_1633} \end{equation}
Write $\cS \subseteq \RR^{\ell}$ for the convex hull of $v_1,\ldots,v_N \in S^{\ell-1}$. By the net property and an elementary convexity argument,
$B^{\ell} \subseteq 2 \cS$. Hence, from (\ref{eq_1633}),
\begin{equation}  \PP \left( G(B^{\ell}) \subseteq 12 \sqrt{n} H \right) \geq \frac{1}{2^N} \geq \frac{1}{2^M}, \label{eq_1636}  \end{equation}
where the last inequality follows from the facts that $N \leq 10^{\ell}$ and $M = n \ell$, combined with our assumption that $\ell \leq c \log n$ for a small universal constant $c > 0$.
Since the standard Gaussian density in $\RR^{n \times \ell}$ is bounded from above by $(2 \pi)^{-M/2}$, 
from (\ref{eq_1637}) and (\ref{eq_1636}),
$$ \Vol( \cB \cap \cD ) = \Vol( \cB_{\ell}(H) ) = \frac{ \Vol( \cB_{\ell}(12 \sqrt{n} H) ) }{(12 \sqrt{n})^{M}}   \geq \frac{\PP \left( G(B^{\ell}) \subseteq 12 \sqrt{n} H
\right)}{(12 \sqrt{n})^{M} \cdot (2 \pi)^{-M/2}}  \geq \left( \frac{c}{\sqrt{n}} \right)^M. $$
On the other hand, $\cD$ is contained in the Hilbert-Schmidt ball
of radius $\sqrt{\ell}$ in $\RR^{n \times \ell} \cong \RR^M$, so
$$
\Vol(\cD)^{1/M}
\leq\sqrt\ell\,\Vol(B^M)^{1/M}
\leq C \sqrt{\ell} \cdot \frac{1}{\sqrt{M}} = \frac{C}{\sqrt{n}}. 
$$
Consequently,
$\Vol(\cB\cap\cD)^{1/M} \geq c \cdot \Vol(\cD)^{1/M}$.
Thus, the maximizing property of $r$ gives
$$
\sqrt r\,\Vol(\cD)^{1/M}
\geq\frac{\Vol(\cB\cap r\cD)^{1/M}}{\sqrt r}
\geq \frac{ \Vol(\cB\cap\cD)^{1/M} }{\sqrt 1}
 \geq c\,\Vol(\cD)^{1/M}.
$$
Consequently $r\geq c$. 
\end{proof}

\appendix \section{Appendix}

The following is a variant of the standard volumetric bound for the cardinality of an $\eps$-net,
see e.g. Pisier \cite[Lemma 4.10]{pisier}.

\begin{lemma} For any $\ell \geq 1$ and $0 < \eps < 1/2$ there exists $\cS \subseteq S^{\ell-1}$
of cardinality at most $2 (5/\eps)^{\ell-1}$ which is an $\eps$-net, i.e., for any $x \in S^{\ell-1}$ there exists
$y \in \cS$ with $|x-y| \leq \eps$.
\label{lem_1714}
\end{lemma}

\begin{proof} For $\ell = 1$ set $\cS = \{-1,1\}$. From now on assume $\ell \geq 2$. Let $\cS \subseteq S^{\ell-1}$ be a maximal $\eps$-separated set. Then the open Euclidean balls
of radius $\eps/2$ that are centered at the points of $\cS$ are pairwise disjoint, while being contained in the spherical shell
$$ \left \{ x \in \RR^\ell \, ; \, 1 - \frac{\eps}{2} < |x| < 1 + \frac{\eps}{2} \right \}. $$
By comparing the volume of the union of these balls with the volume of the shell we obtain
$$ \#(\cS) \cdot (\eps/2)^\ell \leq \left( 1 + \frac{\eps}{2} \right)^{\ell} \, - \, \left( 1 - \frac{\eps}{2} \right)^{\ell}. $$
Since $y^{\ell} - x^{\ell} \leq (y-x) \ell y^{\ell-1}$ for $0 \leq x < y$, we conclude that
$$ \#(\cS) \leq (2 / \eps)^{\ell} \cdot \eps \ell \left( 1 + \frac{\eps}{2} \right)^{\ell-1}
= 2 \ell \left(1 + \frac{2}{\eps} \right)^{\ell-1} \leq 2 \ell \left( \frac{5/2}{\eps} \right)^{\ell-1} \leq 2 \left( \frac{5}{\eps} \right)^{\ell-1}. $$
Since a maximal $\eps$-separated set is an $\eps$-net, the lemma is proven.
\end{proof}

The following is a variant of a standard bound 
for the Dvoretzky dimension of the cube, see e.g.  Ball \cite[Lecture 2]{ball}.

\begin{lemma}
Let $2 \leq \ell \leq n, \eps > 0$ and set $K = [-1,1]^n$. Suppose that an
$\ell$-dimensional linear subspace $E \subseteq \RR^n$ and a centered
ellipsoid $\cF \subseteq E$ satisfy $
\cF \subseteq K \cap E \subseteq (1+\eps)\cF
$.
Then, 
\begin{equation}
 n \geq
 \left( \frac{(1+\eps)^2}{\eps(2+\eps)} \right)^{\frac{\ell-1}{2}}.
 \label{eq_cube_lower_bound}
\end{equation}
\label{lem_cube_lower_bound}
\end{lemma}

\begin{proof}
The section $K \cap E = [-1,1]^n \cap E$ is a polytope with at most $2n$ facets.
Applying an invertible, linear map from $E$ to $\RR^\ell$ which maps $\cF$
to $B^\ell$, we obtain a polytope $P \subseteq \RR^\ell$ with $N \leq 2n$ facets such that
\begin{equation} 
B^\ell \subseteq P \subseteq (1+\eps)B^\ell.
\label{eq_1243} \end{equation}
Write $u_1,\ldots,u_N \in S^{\ell-1}$ for the outer unit normals to
the facets, so that 
$$
P = \bigcap_{j=1}^N
\left\{ x \in \RR^\ell \, ; \, \langle x,u_j\rangle \leq h_j \right\},
$$
with $h_1,\ldots,h_N \geq 1$. As explained e.g. in Ball \cite[Lecture 2]{ball},
it follows from (\ref{eq_1243}) that  the
spherical caps
$$
C_j = \left\{ x \in S^{\ell-1} \, ; \,
\langle x,u_j\rangle \geq 1/(1+\eps) \right\}
\qquad (j=1,\ldots,N)
$$
cover the entire sphere $S^{\ell-1}$.  Write $\sigma$ for the uniform probability measure on $S^{\ell-1}$.
We will show that each cap $C_j$ has measure at most $s^{\ell-1}/2$
for $s = \sqrt{1 - (1+\eps)^{-2}}$.
Indeed, an orthogonal transformation maps $C_j$ to the 
set of all points of the form
$(y, \sqrt{1-|y|^2})$ with $y \in sB^{\ell-1}$. The surface area
element in these coordinates is $(1-|y|^2)^{-1/2}\,dy$. Thus,
\begin{align}
\sigma(C_j)
&= \frac{\displaystyle
       \int_{sB^{\ell-1}} \frac{dy}{\sqrt{1-|y|^2}}}
      {\displaystyle
       2\int_{B^{\ell-1}} \frac{dy}{\sqrt{1-|y|^2}}} = \frac{s^{\ell-1}\displaystyle
       \int_{B^{\ell-1}} \frac{dy}{\sqrt{1-s^2|y|^2}}}
      {\displaystyle
       2\int_{B^{\ell-1}} \frac{dy}{\sqrt{1-|y|^2}}}
\leq \frac{s^{\ell-1}}{2}. \label{eq_1457}
\end{align}
Since these $N$ caps cover the sphere we have $1 \leq N s^{\ell-1}/2 \leq n s^{\ell-1}$, 
and (\ref{eq_cube_lower_bound}) follows.
\end{proof}

\bigskip
\noindent School of Mathematical Sciences, Tel Aviv University, Tel Aviv, Israel; and \\
Department of Mathematics, Weizmann Institute of Science, Rehovot, Israel. \\
\textit{e-mail:} \texttt{klartagb@tau.ac.il}

\bigskip
\noindent School of Mathematical Sciences, Tel Aviv University, Tel Aviv, Israel. \\
\textit{e-mail:} \texttt{shaharmoshe2@tau.ac.il}



\begin{thebibliography}{99}

\bibitem{ball}
Ball, K., {\it An elementary introduction to modern convex geometry.}
\newblock In {Flavors of geometry}, volume 31 of
{Math. Sci. Res. Inst. Publ.}, pages 1--58.
Cambridge Univ. Press,  1997.

\bibitem{Bourgain}
Bourgain, J., Lindenstrauss, J.,
{\it Almost Euclidean sections in spaces with a symmetric basis.}
\newblock In {Geometric aspects of functional analysis}, Israel seminar.
Lecture Notes in Math., Vol. 1376, Springer, (1989), 278--288.

\bibitem{BLYZ}
B\"or\"oczky, K. J., Lutwak, E., Yang, D., Zhang, G.,
{\it The log-Brunn-Minkowski inequality.}
\newblock Adv. Math., Vol. 231, no. 3--4, (2012), 1974--1997.

\bibitem{BLYZ2}
B\"or\"oczky, K. J., Lutwak, E., Yang, D., Zhang, G.,
{\it The logarithmic Minkowski problem.} 
\newblock J. Amer. Math. Soc., Vol. 26, no. 3, (2013),  831--852.


\bibitem{BLL}
Brascamp, H. J., Lieb, E. H., Luttinger, J. M.,
{\it A general rearrangement inequality for multiple integrals.}
\newblock J. Functional Analysis, Vol. 17, no. 2, (1974), 227--237.

\bibitem{BBT}
Burago, D., Ivanov, S., Tabachnikov, S.,
{\it Topological aspects of the Dvoretzky theorem.}
\newblock J. Topol. Anal., Vol. 2, no. 4, (2010), 453--467.

\bibitem{DvoretzkyOriginalAnnouncement}
Dvoretzky, A.,
{\it A theorem on convex bodies and applications to Banach spaces.}
\newblock Proc. Nat. Acad. Sci. U.S.A., Vol. 45, no. 2,
(1959), 223--226.

\bibitem{DvoretskyOriginal}
Dvoretzky, A., {\it Some results on convex bodies and Banach spaces.}
\newblock In {Proc. Internat. Sympos. Linear Spaces (Jerusalem, 1960)},
pages 123--160. Jerusalem Academic Press, Jerusalem;
Pergamon, Oxford, 1961.

\bibitem{DvoretskyErratum}
Dvoretzky, A., {\it Errata.}
\newblock Proc. Nat. Acad. Sci. U.S.A., Vol. 45, no. 10,
(1959), 1554.

\bibitem{figiel}
Figiel, T.,
{\it A short proof of Dvoretzky's theorem on almost spherical
sections of convex bodies.}
\newblock Compositio Math., Vol. 33, no. 3, (1976), 297--301.

\bibitem{Fresen}
Fresen, D. J.,
{\it Explicit Euclidean embeddings in permutation invariant normed spaces.}
\newblock Adv. Math., Vol. 266, (2014), 1--16.

\bibitem{G}
Gordon, Y., {\it Some inequalities for Gaussian processes and applications.}
\newblock Israel J. Math., Vol. 50, no. 4, (1985), 265--289.

\bibitem{GLPT}
Gordon, Y., Litvak, A. E., Pajor, A., Tomczak-Jaegermann, N.,
{\it Random $\varepsilon$-nets and embeddings in $\ell^N_{\infty}$.}
\newblock Studia Math., Vol. 178, no. 1, (2007), 91--98.

\bibitem{Grothendieck}
Grothendieck, A.,
{\it Sur certaines classes de suites dans les espaces de Banach,
et le th\'eor\`eme de Dvoretzky-Rogers.}
\newblock Resenhas, Vol. 3, no. 4, (1998), 447--477.
\newblock With a foreword by Paulo Cordaro. Reprint of the 1953 original.

\bibitem{Haddad} Haddad, J., {\it A Rogers–Brascamp–Lieb–Luttinger inequality in the space of matrices. } 
In Geometry, Analysis and Convexity, Real Sociedad Matem\'atica Espa\~nola (RSME) Series, Vol. 17, Springer, (2026), 107--127.

\bibitem{KashinSzarek}
Kashin, B. S., Szarek, S. J.,
{\it The Knaster problem and the geometry of high-dimensional cubes.}
\newblock C. R. Math. Acad. Sci. Paris, Vol. 336, no. 11,
(2003), 931--936.

\bibitem{KlartagLehecSlicing}
Klartag, B., Lehec, J.,
{\it Affirmative resolution of Bourgain's slicing problem using
Guan's bound.}
\newblock Geom. Funct. Anal., Vol. 35, no. 4, (2025), 1147--1168.

\bibitem{KN}
Klartag, B., Novikov, T.,
{\it A weak version of the $\varepsilon$-Dvoretzky conjecture
for normed spaces.}
\newblock Israel J. Math., (2026), 12 pp.

\bibitem{LatalaMatlak}
Lata{\l}a, R., Matlak, D.,
{\it Royen's proof of the Gaussian correlation inequality.}
\newblock In {Geometric aspects of functional analysis}, Israel seminar.
Lecture Notes in Math., Vol. 2169, Springer, (2017), 265--275.


\bibitem{Mil71}
Milman, V. D.,
{\it A new proof of A. Dvoretzky's theorem on cross-sections of convex bodies.}
\newblock Funkcional. Anal. i Prilo\v{z}en., Vol. 5, no. 4,
(1971), 28--37.

\bibitem{MilmanTopology}
Milman, V. D.,
{\it A few observations on the connections between local theory
and some other fields.}
\newblock In {Geometric aspects of functional analysis}, Israel seminar.
Lecture Notes in Math., Vol. 1317, Springer, (1988), 283--289.

\bibitem{MilmanPajor}
Milman, V. D., Pajor, A.,
{\it Entropy and asymptotic geometry of non-symmetric convex bodies.}
\newblock Adv. Math., Vol. 152, no. 2, (2000), 314--335.

\bibitem{ConjectureStatement}
Milman, V. D., {\it Dvoretzky's theorem -- thirty years later.}
\newblock Geom. Funct. Anal. (GAFA), Vol. 2, no. 4, (1992), 455--479.


\bibitem{EMilmanMeanWidth}
Milman, E.,
{\it On the mean-width of isotropic convex bodies and their associated
$L_p$-centroid bodies.}
\newblock Int. Math. Res. Not. (IMRN), no. 11, (2015), 3408--3423.

\bibitem{PV}
Paouris, G., Valettas, P.,
{\it Dichotomies, structure, and concentration in normed spaces.}
\newblock Adv. Math., Vol. 332, (2018), 438--464.

\bibitem{pisier}
Pisier, G., {\it The volume of convex bodies and Banach space geometry.}
\newblock Volume 94 of {Cambridge Tracts in Math.}
Cambridge Univ. Press,  1989.

\bibitem{rogers}
Rogers, C. A., {\it Covering a sphere with spheres.}
\newblock Mathematika, Vol. 10, no. 2, (1963), 157--164.

\bibitem{rogers2}
Rogers, C. A., {\it A single integral inequality.}
\newblock J. London Math. Soc., Vol. 32, no. 1,
(1957), 102--108.

\bibitem{schechtman}
Schechtman, G., {\it Euclidean sections of convex bodies.}
\newblock In {Asymptotic geometric analysis}, {Fields Inst. Commun.}, Vol. 68, Springer, (2013), 271--288.

\bibitem{Szan}
Szankowski, A.,
{\it On Dvoretzky's theorem on almost spherical sections of convex bodies.}
\newblock Israel J. Math., Vol. 17, (1974), 325--338.

\bibitem{Tikhomirov}
Tikhomirov, K. E.,
{\it Almost Euclidean sections in symmetric spaces and concentration
of order statistics.}
\newblock J. Funct. Anal., Vol. 265, no. 9, (2013), 2074--2088.


\end{thebibliography}
\end{document}